\documentclass[11pt,a4paper]{article}
\usepackage{fancyhdr}
\usepackage{amsmath}
\usepackage{amsthm}
\usepackage{amssymb}
\usepackage{wasysym}
\usepackage{paralist}
\usepackage{makeidx}
\usepackage[all]{xy}
\usepackage{mathdots}
\usepackage{yhmath}
\usepackage{mathtools}

\usepackage{young}
\usepackage[vcentermath]{youngtab}

\usepackage{graphicx}
\usepackage{epstopdf}

\input xy

\newcommand{\nc}{\newcommand}
 
 \nc{\cl}{\centerline}

 \nc{\SL}{{\rm SL}}
 \nc{\hatQ}{{\hat Q}}
 \nc{\sgn}{{\rm sgn}}
 \nc{\ad}{{\rm ad}}
 
  \nc{\idx}{{\rm index}}
 
 \nc{\Mat}{{\rm Mat}}
 
  \nc{\Ann}{{\rm Ann}}
  
  \nc{\A}{{\mathcal A}}
  
    \nc{\B}{{\mathcal B}}
    
      \nc{\C}{{\mathcal C}}
      
           \nc{\DD}{{\mathcal D}}
 
 \nc{\Loewy}{{\rm Loewy}}
 
 \nc{\Supp}{{\rm Supp}}
 
 \nc{\env}{{\rm env}}

 \nc{\wt}{{\rm wt}}
 
 \nc{\poll}{{\rm \l}}
 
 \nc{\efp}{{\Bbb F}_p}

 \newcommand{\Par}{{\rm Par}}
 
\nc{\baru}{{\overline u}}

\nc{\baralpha}{{\overline \alpha}}

\nc{\bargamma}{{\overline \gamma}}

\nc{\barA}{{\bar A}}

\nc{\barq}{{\overline q}}

 \nc{\Orb}{{\rm Orb}}

 \nc{\hatsigma}{{\hat \sigma}}

 \nc{\hatpi}{{\hat \pi}}
 
  \nc{\hatzeta}{{\hat \zeta}}

 \nc{\Ocal}{{\mathcal O}}

 \nc{\M}{\mathfrak{m}}
 
 \nc{\seee}{\mathbb C}
 
 \nc{\varleq}{\preccurlyeq}

 \nc{\hatlambda}{{\hat\lambda}}
 
 \nc{\hatphi}{{\hat \phi}}
 
 \nc{\daggerlambda}{{\lambda^\dagger}}

    \nc{\barr}{{\bar r}}
  \nc{\bart}{{\bar t}}
  
    \nc{\barsigma}{{\bar \sigma}}

   \newcommand{\op}{{\rm op}}

\nc\diag{{\rm diag}}
\renewcommand{\vert}{{\,|\,}}
\nc{\hatL}{{\hat L}}

\nc{\barE}{{\bar   E}}
\nc{\D}{{\mathcal D}}
\nc{\E}{{\mathcal E}}
\nc{\F}{{\mathcal F}}
\nc{\FF}{{\mathcal F}}
\nc{\I}{{\mathcal I}}
\nc{\even}{{\rm e}}
\nc{\ep}{\epsilon}
\nc{\odd}{{\rm o}}
\nc{\Coker}{{\rm Coker}}
\nc{\olE}{{\overline E}}
\nc{\indBG}{{\rm ind}_B^G\,}
\nc{\indHG}{{\rm ind}_H^G\,}
\renewcommand{\O}{{\mathcal O}}

\nc{\que}{{\mathbb Q}}
\nc{\barlambda}{{\bar\lambda}}
\nc{\barmu}{{\bar\mu}}
\nc{\barnu}{{\bar\nu}}
\nc{\bartau}{{\bar\tau}}
\nc{\barm}{{\overline  m}}
\nc{\divind}{{\rm div.ind}}
\nc{\tl}{{\tilde{\lambda}}}
\nc{\dar}{\downarrow}

\nc{\eno}{{\mathbb N}_0}

\nc{\Sym}{{\rm Sym}}
\nc{\Symm}{{\rm Sym}}

\newcommand{\q}{\quad}
\newcommand{\de}{\delta}

\renewcommand{\mod}{{\rm mod}}

\newcommand{\Sp}{{\rm Sp}}

\newcommand{\bs}{\bigskip}

\renewcommand{\vert}{\,|\,}

\renewcommand{\sgn}{{\rm sgn}}

\xyoption{all}

\renewcommand{\vert}{\,|\,}
 
\newcommand{\zed}{{\mathbb Z}}

\newcommand{\End}{{\rm End}}
\newcommand{\Hom}{{\rm Hom}}

\renewcommand{\mod}{{\rm mod}}

\nc{\barB}{{\overline B}}
\nc{\barb}{{\overline b}}

\renewcommand{\mod}{{\rm{mod}}}

\nc{\geom}{{\rm geom}}
\nc{\rep}{{\rm rep}}

\newtheorem{definition}{Definition}[section]
\newtheorem{proposition}[definition]{Proposition}
\newtheorem{theorem}[definition]{Theorem}
\newtheorem{lemma}[definition]{Lemma}
\newtheorem{corollary}[definition]{Corollary}
\newtheorem{example}[definition]{Example}

\newtheorem{remark}[definition]{Remark}

\begin{document}


\centerline{\bf  Cellularity of endomorphism algebras of}

\centerline{ \bf  Young permutation modules.}

\bigskip

\centerline{Stephen Donkin}

\bigskip

{\it Department of Mathematics, University of York, York YO10 5DD}

\medskip

\centerline{\tt stephen.donkin@york.ac.uk}

\bs

\centerline{3  June    2020}

\bs

\section*{Abstract}

\q Let $E$ be an $n$-dimensional vector space. Then the symmetric group $\Sym(n)$ acts on $E$ by permuting the elements of a basis and hence on the $r$-fold tensor product $E^{\otimes r}$.  Bowman, Doty and Martin ask,  in \cite{BDM},  whether the endomorphism algebra $\End_{\Sym(n)}(E^{\otimes r})$ is cellular.  The module $E^{\otimes r}$ is the permutation module for a certain Young $\Sym(n)$-set.  We shall show that the endomorphism algebra of  the permutation module on an arbitrary  Young $\Sym(n)$-set  is a cellular algebra.   We   determine, in terms of the point stabilisers which appear, when the endomorphism algebra  is quasi-hereditary. 

\section{Introduction}

\q We fix a positive integer $n$. The symmetric group of degree $n$ is denoted $\Sym(n)$.   For a partition $\lambda=(\lambda_1,\lambda_2,\ldots)$ of $n$ we have the Young subgroup, i.e. the group  $\Sym(\lambda)=\Sym(\lambda_1)\times \Sym(\lambda_2)\times \cdots$, regarded as a subgroup of $\Sym(n)$ in the usual way. By a Young $\Sym(n)$-set we mean a finite $\Sym(n)$-set such that each point stabiliser is conjugate to a Young subgroup.  Let $R$ be a commutative ring. Our interest is in the endomorphism algebra $\End_{\Sym(n)} (R\,\Omega)$ of the permutation module $R\,\Omega$ on a Young $\Sym(n)$-set $\Omega$.   We shall show that $\End_{\Sym(n)}(\zed\, \Omega)$ has a cellular structure, Theorem 6.4, hence by base change so has $\End_{\Sym(n)}(R\, \Omega)$, for an arbitrary commutative ring $R$.

\q Taking the base ring now to be a field $k$ of positive characteristic, we give a criterion for $\End_{\Sym(n)}(k\Omega)$ to be a quasi-hereditary algebra, in terms of the set of partitions $\lambda$ of $n$ for which $\Sym(\lambda)$ occurs as a point stabiliser,  and the characteristic $p$ of $k$, see Theorem 6.4.  This is applied to the case $\Omega=I(n,r)$, the set of maps from $\{1,\ldots,r\}$ to $ \{1,\ldots,n\}$, for a positive integer $r$, with $\Sym(n)$ acting by composition of maps.  The permutation module $k I(n,r)$ may be regarded as the $r$th tensor power $E^{\otimes r}$ of an $n$-dimensional vector space $E$, and we thus determine when $\End_{\Sym(n)}(E^{\otimes r})$ is quasi-hereditary, see Proposition 7.3. 

\q Our procedure is to analyse the endomorphism algebra of a Young permutation module  in the spirit of the Schur algebra $S(n,r)$ (which is a  special case).  Of particular importance to us will be the fact that the Schur algebras is  quasi-hereditary.  There are several approaches to this (see 
e.g. \cite[Section A5]{q-Donk} and \cite{Parshall})  but for us the most convenient is that of Green, \cite{CASA}.  This has the advantage of being a purely combinatorial account carried out over an arbitrary commutative base ring.  So we regard what follows as a modest generalisation of  some aspects of   \cite{CASA}: we follow  Green's approach and notation to a large extent.

\section{Preliminaries}

\bs

\q We write $\mod(S)$ for the category of finitely generated modules over a ring $S$.

\q Let $G$ be a finite group and $K$ a field of characteristic $0$. Let $X$ be a finitely generated $KG$-module. Suppose that all composition factors of $X$ are absolutely irreducible.  Let $U_1,\ldots,U_d$ be a complete set of pairwise non-isomorphic composition factors of $X$.  We write $X$ as a direct sum of simple modules $X=X_1\oplus \cdots \oplus  X_h$.   For $1\leq  i \leq d$ let $m_i$ be the number of  elements $r\in \{1,\ldots,h\}$ such that $X_r$ is isomorphic to $U_i$.  Let $S=\End_G(X)$. Then $S$ is isomorphic to the product of the matrix algebras \\
$M_{m_1}(K), \ldots,  M_{m_d}(K)$.  Let the corresponding irreducible modules for $S$ be $L_1,\ldots,L_d$.   We have an exact functor from $f:\mod(KG)\to \mod(S)$,  given on objects by $f(Z)=\Hom_{\Sym(n)}(X,Z)$.  Moreover we have $S=f(X)=\bigoplus_{r=1}^h \Hom_G(X,X_r)$. If follows that the modules    $L_i=  fU_i=\Hom_G(X,U_i)$, $1\leq i\leq d$, form a complete set of pairwise non-isomorphic irreducible $S$-modules. 

\q The situation in positive characteristic is similar, cf. \cite[(3.4) Proposition]{GreenMER}.   Suppose now that $F$ is any  field  which is a  splitting field for $G$.  Let  $Y$ be  a finitely generated $KG$-module such that every indecomposable component is absolutely indecomposable. Let $V_1,\ldots,V_e$ be a complete set of pairwise non-isomorphic indecomposable summands of $Y$. We write $Y$ as a direct sum of indecomposable modules $Y=Y_1\oplus\cdots \oplus Y_k$.  For $1\leq  j \leq e$ let $n_j$ be the number of  elements $r\in \{1,\ldots,k\}$ such that $X_r$ is isomorphic to $V_j$.  Let $T=\End_G(Y)$.  Then each $P_j=\Hom_G(Y,V_j)$ is naturally a $T$-module and the modules $P_1,\ldots,P_e$ form a complete set of pairwise non-isomorphic projective $T$-modules. Let $N_j$ be the head of $P_j$, $1\leq j\leq e$. Then the modules $N_1,\ldots,N_e$ form a complete set of pairwise non-isomorphic irreducible $T$-modules. The dimension of $N_j$ over $F$ is $n_j$.

\q We now  fix a positive integer  $n$.  We write $\Par(n)$ for the set of partitions of $n$.  By the support $\zeta(\Omega)$ of a Young $\Sym(n)$-set $\Omega$ we mean the set of $\lambda\in \Par(n)$ such that the Young subgroup $\Sym(\lambda)$ is a point stabiliser.   Let $R$ be a commutative ring.  For a Young $\Sym(n)$-set $\Omega$ we write $S_{\Omega,R}$ for the endomorphism algebra $\End_{\Sym(n)}(R\Omega)$ of the permutation module $R\Omega$.    For $\lambda\in \Par(n)$ we write $M(\lambda)_R$ for the permutation module   $R\, \Sym(n)/\Sym(\lambda)$.   

\q We have the usual dominance partial order  $\trianglelefteq$  on $\Par(n)$.  Thus, for $\lambda=(\lambda_1,\lambda_2,\ldots),  \mu=(\mu_1\mu_2,\ldots)\in \Par(n)$,  we write   $\lambda \trianglelefteq \mu$ if $\lambda_1+ \cdots+\lambda_a\leq \mu_1+\cdots+\mu_a$ for all $1\leq a\leq n$.

\q Recall that the Specht modules $\Sp(\lambda)_\que$, $\lambda\in \Par(n)$, form a complete set of pairwise irreducible $\que \Sym(n)$-modules.  For 
$\lambda\in \Par(n)$ we have 
$M(\lambda)_\que=\Sp(\lambda)_\que \oplus C$, where $C$ is a direct sum of modules of the form $\Sp(\mu)$  with $ \lambda \vartriangleleft \mu$,   and moreover every Specht module $\Sp(\mu)_\que$ with $\lambda \triangleleft \mu$ occurs in $C$ (see for example  \cite[14.1]{James}).

\q For a Young $\Sym(n)$-set $\Omega$ we define  
$$\zeta^\trianglerighteq(\Omega)=\{\mu\in \Par(n) \vert \mu \trianglerighteq \lambda  \hbox{ for some } \lambda\in \zeta(\Omega)\}.$$

\q Thus the composition factors of $\que\Omega$ are $\{\Sp(\mu)_\que \vert \mu \in \zeta^\trianglerighteq(\Omega)\}$ and, setting $\nabla_\Omega(\lambda)_\que=\Hom_{\Sym(n)}(\que \Omega,\Sp(\mu)_\que)$, we have the following. 

\begin{lemma} The modules $\nabla_\Omega(\lambda)_\que$,   $\lambda\in  \zeta^\trianglerighteq(\Omega)$,  form a complete set of pairwise non-isomorphic irreducible $S_{\Omega,\que}$-modules. 
\end{lemma}

\begin{remark} Since $S_{\Omega,\que}$ is a direct sum of matrix algebras over $\que$ it is semisimple, all irreducible modules  are absolutely irreducible and  \\
 $\dim_\que  S_{\Omega,\que}=\sum_{\lambda\in   \zeta^\trianglerighteq(\Omega)} (\dim_\que \nabla_\Omega(\lambda)_\que)^2.$

\end{remark}

\q We now let $k$ be a field of characteristic $p>0$.  For $\lambda\in \Par(n)$ we have  the Young module $Y(\lambda)$ for $k\Sym(n)$, labelled by $\lambda$, as described in \cite[Section 4.4]{q-Donk} for example.   Then we have $M(\lambda)_k=Y(\lambda)\oplus C$, where $C$ is a direct sum of Young modules $Y(\mu)$, with $\lambda \triangleleft \mu$, see for example \cite[Section 4.4 (1) (v)]{q-Donk}.     A partition $\lambda=(\lambda_1,\lambda_2,\ldots) $ will be called $p$-restricted (also called column $p$-regular) if $\lambda_i-\lambda_{i+1} < p$ for all $i\geq 1$.  A partition $\lambda$ has a unique expression
$$\lambda=\sum_{i\geq 0}  p^i \lambda(i)$$ where each $\lambda(i)$ is a $p$-restricted partition.   This is called the base $p$ (or $p$-adic) expansion of $\lambda$.   

\q We write $\Lambda(n)$ for the set of all $n$-tuples of non-negative integers.  An expression $\lambda=\sum_{i\geq 0} p^i\gamma(i)$, with all $\gamma(i)\in \Lambda(n)$ (but not necessarily restricted) will be called a weak $p$ expansion. 

\q For an $n$-tuple of non-negative integers $\gamma$ we write $\bargamma$ for the partition obtained by arranging the entries in descending order.

\begin{definition} For $\lambda,\mu\in \Par(n)$ we shall say that $\mu$ $p$-dominates $\lambda$, and write $\mu   \trianglerighteq_p \lambda$ (or  $\lambda  \trianglelefteq_p \mu$)  if there exists  a weak $p$  expansion $\lambda=\sum_{i\geq 0} p^i \gamma(i)$,  such that $\mu(i)  \trianglerighteq  {\overline {\gamma(i)}}   $ for all $i\geq 0$,  where $\mu=\sum_{i\geq 0} p^i\mu(i)$ is the base $p$ expansion of $\mu$. 
\end{definition}
\q Note that $\lambda  \trianglelefteq_p \mu$ implies  $\lambda  \trianglelefteq  \mu$.    

\q By  \cite[Section 3, Remark]{DTilt},    for $\lambda,\mu\in \Par(n)$,  then module $Y(\mu)$ appears as a component of $M(\lambda)_k$ if and only if  $\lambda  \trianglelefteq_p \mu$. For a Young $\Sym(n)$-set $\Omega$ we define
  $$\zeta^{ \trianglerighteq_p}(\Omega)=\{\mu\in \Par(n) \vert \mu \trianglerighteq_p \lambda  \hbox{ for some }  \lambda\in \zeta(\Omega)\}.$$
  
  \q   Writing $P(\mu)=\Hom_{\Sym(n)}(k\Omega,Y(\mu))$ and writing $L(\mu)$ for the head of $P(\lambda)$,   for $\mu\in \zeta^{ \trianglerighteq_p}(\Omega)$ we have the following.

\begin{lemma} The modules $L(\mu)$,   $\mu \in  \zeta^{\trianglerighteq_p} (\Omega)$,  form a complete set of pairwise non-isomorphic irreducible $S_{\Omega,k}$-modules. 
\end{lemma}

\section{Basic Constructions}

\q We fix a positive integer $n$ and a Young $\Sym(n)$-set $\Omega$. Here we assume the base ring $R$ is either the ring  integers $\zed$ or the field of rational numbers $\que$.  We write $M_{\Omega,R}$, or just $M_R$  for the permutation module $R\,\Omega$ over $R\Sym(n)$. We also just write $M$ for $M_{\Omega,\zed}$.  We shall sometimes write simply $S_R$ for $S_{\Omega,R}$ and just $S$ for $S_\zed$.   We identify $S$ with a subring or $S_\que$ in the natural way.

\q Let     $\{ \O_\alpha\vert \alpha\in \Lambda_\Omega\}$ be a complete set of  orbits in  $\Omega$.   For $\lambda\in \zeta(\Omega)$ we pick $\alpha(\lambda) \in \Lambda_\Omega$ such that $\Sym(\lambda)$ is a point stabiliser of some element of $\O_\alpha$.

\q We put $M_{\alpha,R}=R\,\O_\alpha$, and sometimes write just $M_\alpha$ for $M_{\alpha,\zed}$,  for $\alpha\in \Lambda_\Omega$. For $\beta\in \Lambda_\Omega$ we define the element    $\xi_\beta$ of $S_R$ to be the projection onto $M_{\beta,R}$  coming from the decomposition $M_R=\bigoplus_{\alpha\in \Lambda_\Omega} M_{\alpha,R}$.  Then each $\xi_\alpha$ is  idempotent and we have the orthogonal decomposition:
$$1_S=\sum_{\alpha\in \Lambda_\Omega} \xi_\alpha.$$

\q For a left  $S_R$-module $V$ and $\beta\in \Lambda_\Omega$ we have the
$\beta$ weight space ${}^\beta V=\xi_\beta V$ and the weight space decomposition 
$$V=\bigoplus_{\alpha\in \Lambda_\Omega} {}^\alpha V.$$
For $\lambda\in \Par(n)$ we define
$${}^\lambda V=\begin{cases} \xi_{\alpha(\lambda)}V, & \hbox{ if } \lambda\in \zeta(\Omega);\cr
0, & \hbox{otherwise.}
\end{cases}
$$
Similar remarks apply to  weight spaces of right $S_R$-modules.

\begin{lemma} Let $\lambda\in \zeta^\trianglerighteq(\Omega)$.  Then

(i) $ \dim_\que  {}^\lambda \nabla_\Omega(\lambda)_\que =1$; and 

(ii) if $\mu\in \Par(n)$ and ${}^\mu \nabla_\Omega(\lambda)_\que\neq 0$ then $\mu \trianglelefteq\lambda$.

\end{lemma}

\begin{proof} Let $\mu\in \Par(n)$ and suppose ${}^\mu \nabla_\Omega(\lambda)_\que\neq 0$.  Thus \\
$\xi_\mu \Hom_{\Sym(n)}(M_\que,\Sp(\lambda)_\que)\neq 0$ i.e. $\Hom_{\Sym(n)}(M(\mu)_\que,\Sp(\lambda)_\que)\neq 0$ and so $\mu  \trianglerighteq \lambda$, giving (ii).  Moreover 
$$\xi_\lambda \Hom_{\Sym(n)}(M_\que,\Sp(\lambda)_\que) =\Hom_{\Sym(n)}(M(\lambda)_\que,\Sp(\lambda)_\que)= \que$$
giving (i).
\end{proof}

\q For $\lambda\in \Par(n)$ we set
$$\xi_\lambda=\begin{cases} \xi_{\alpha(\lambda)}, &\hbox{ if } \lambda\in \zeta(\Omega):\cr
0, & \hbox{ otherwise.}
\end{cases}
$$

\q For $\lambda\in \Par(n)$ we set $S_R(\lambda)=S_R\xi_\lambda S_R$ and for $\sigma\subseteq \Par(n)$ set
$$S_R(\sigma)=\sum_{\lambda\in \sigma} S_R(\lambda).$$

\q We also write simply $S(\lambda)$ for $S_\zed(\lambda)$ and $S(\sigma)$ for $S_\zed(\sigma)$. 

\q Let $\leq$ be a partial order on $\Par(n)$ which is a refinement of the dominance partial order.  For $\lambda\in \zeta(\Omega)$ we set  
$S_R(\geq\lambda)=S_R(\sigma)$, where \\
$\sigma=\{\mu\in \Par(n) \vert \mu\geq \lambda\}$, and 
$S_R(>\lambda)=S_R(\tau)$, where  \\
$\tau=\{\mu\in \Par(n)  \vert \mu > \lambda\}$.   Thus 
$$S_R(\geq \lambda)=S_R\xi_\lambda S_R+S(>\lambda).$$

\q We set $V_R(\lambda)=S_R(\geq\lambda)/S_R(>\lambda)$. So we have 
$$V_R(\lambda)^\lambda=(S_R\xi_\lambda +S_R(>\lambda))/S_R(>\lambda),$$
$${}^\lambda V_R(\lambda)=(\xi_\lambda S_R +S_R(>\lambda))/S_R(>\lambda)$$
and the multiplication map $S_R\xi_\lambda \times  \xi_\lambda S_R \to S_R$ induces a surjective map 
$$\phi_R(\lambda): V_R(\lambda)^\lambda \otimes_R {}^\lambda V_R(\lambda)\to V_R(\lambda).$$

\q For left $S_R$-modules $P,Q$ and $\lambda\in \Par(n)$ we define $\Hom_{\Sym(n)}^\lambda (P,Q)$ to be the $R$-submodule of $\Hom_{\Sym(n)}(P,Q)$ spanned by all composite maps $f\circ g$, with 
$f\in \Hom_{\Sym(n)}(M(\lambda)_R,Q)$ and  $g\in \Hom_{\Sym(n)}(P,M(\lambda)_R)$.    For a subset $\sigma$ of $\Par(n)$ we set
$$\Hom_{\Sym(n)}^\sigma(P,Q)=\sum_{\lambda\in \sigma} \Hom_{\Sym(n)}^\lambda (P,Q).$$

\q We note some similarity of our approach here via these groups of homomorphisms  with  the approach to Schur algebras due to Erdmann, \cite{ErdmannStrat} via stratification.

\q For $\lambda\in \Par(n)$ we define $\Hom_{\Sym(n)}^{\geq\lambda}(P,Q)=\Hom_{\Sym(n)}^\sigma(P,Q)$, where $\sigma=\{\mu\in \Par(n)\vert \mu\geq \lambda\}$, and $\Hom_{\Sym(n)}^{>\lambda}(P,Q)=\Hom_{\Sym(n)}^\tau(P,Q)$, where $\tau=\{\mu\in \Par(n)\vert \mu>  \lambda\}$. 

\bs

\q Note that if  $\lambda\not\in \zeta(\Omega)$ then $V_R(\lambda)=0$.   Suppose $\lambda\in \zeta(\Omega)$. Then we have
\begin{align*}S_R\xi_\lambda S_R&=\sum_{\alpha,\beta,\gamma,\delta\in \Lambda_\Omega} \Hom_{\Sym(n)}(M_{\alpha,R},M_{\beta,R})\xi_\lambda \Hom_{\Sym(n)}(M_{\gamma,R},M_{\delta,R})\cr
&=\sum_{\alpha,\delta\in \Lambda_\Omega} \Hom_{\Sym(n)}(M_{\alpha,R},M_{\alpha(\lambda)})\xi_\lambda 
 \Hom_{\Sym(n)}(M_{\alpha(\lambda)},M_{\delta,R})\cr
 &= \bigoplus_{\alpha,\beta\in \Lambda_\Omega} \Hom_{\Sym(n)}^\lambda (M_{\alpha,R},M_{\beta,R})
 \end{align*}
 and hence 
 $$S_R(\sigma)=\bigoplus_{\alpha,\beta\in \Lambda_\Omega} \Hom_{\Sym(n)}^\sigma (M_{\alpha,R},M_{\beta,R})\eqno{(1)}$$
 for $\sigma\subseteq \Par(n)$. In particular we have 
  $$S_R(\geq\lambda)=\bigoplus_{\alpha,\beta\in \Lambda_\Omega} \Hom_{\Sym(n)}^{\geq\lambda} (M_{\alpha,R},M_{\beta,R})$$
  and 
  $$S_R(>\lambda)=\bigoplus_{\alpha,\beta\in \Lambda_\Omega} \Hom_{\Sym(n)}^{>\lambda} (M_{\alpha,R},M_{\beta,R})$$
and hence 
$$V_R(\lambda)=\bigoplus_{\alpha,\beta\in \Lambda_\Omega} \Hom_{\Sym(n)}^{\geq\lambda} (M_{\alpha,R},M_{\beta,R})/\Hom_{\Sym(n)}^{>\lambda} (M_{\alpha,R},M_{\beta,R}).\eqno{(2)}$$

\begin{example}  Of crucial importance is the motivating example  of the usual Schur algebra $S(n,r)$.   Let $R$ be a commutative ring and let $E_R$ be a free $R$-module of rank $n$.  Then $\Sym(r)$ acts on the $r$-fold tensor product  $E_R^{\otimes r}=E_R\otimes \cdots \otimes_R E_R$  by place permutation,  and the Schur algebra $S_R(n,r)$ may be realised as $\End_{\Sym(r)}(E_R^{\otimes r})$.  

\q We choose an $R$-basis $e_1,\ldots,e_n$ of $E_R$.  We  write $I(n,r)$ for the set of maps from $\{1,\ldots,r\}$ to $\{1,\ldots,n\}$. We regard $i\in I(n,r)$ as an $r$-tuple of elements $(i_1,\ldots, i_r)$ with entries in $\{1,\ldots, n\}$ (where $i_a=i(a)$, $1\leq a\leq r$).   The group $\Sym(r)$ acts on $I(n,r)$ composition of maps, i.e.  by $w\cdot i=i\circ w^{-1}$, for $w\in \Sym(r)$, $i\in I(n,r)$.   Moreover, for $i\in I(n,r)$, $w\in \Sym(r)$, we have $w\cdot e_i=e_{i\circ w^{-1}}$. 

\q  We may thus regard $E_R^{\otimes r}$ as the $R \Sym(r)$ permutation module $R \Omega$ on $\Omega=I(n,r)$.   Note that $\zeta(\Omega)=\Lambda^+(n,r)$, the set of partitions of $r$ with at most $n$ parts.   We write $\Lambda(n,r)$ for the set of weights, i.e. the set of  $n$-tuples of non-negative integers  $\alpha=(\alpha_1\ldots,\alpha_n)$ such that $\alpha_1+\cdots+\alpha_n=r$.  An element $i$ of $I(n,r)$ has weight $\wt(i)=(\alpha_1,\ldots,\alpha_n)\in \Lambda(n,r)$,  where $\alpha_a=|i^{-1}(a)|$, for $1\leq a\leq n$.  For $\alpha\in \Lambda(n,r)$ we have the orbit $\O_\alpha$ consisting or all $i\in I(n,r)$ such that $\wt(i)=\alpha$. Then  $R \Omega = \bigoplus_{\alpha\in \Lambda(n,r)}  R \O_\alpha$.
\end{example}

\section{Groups of homomorphisms between Young permutation modules}

\q In the situation of the Example 3.2  it follows from the quasi-hereditary structure of $S_\zed(n,r)$ that $V_\zed(\lambda)$ is a free abelian group - indeed an explicit basis is given by Green in \cite[(7.3) Theorem, (ii),(iii)]{CASA}.  Thus,  taking $r=n$,  from Section 3, (2),    we have the following.

\begin{lemma} For all $\lambda,\mu,\tau\in \Par(n)$ the quotient  
$$ \Hom_{\Sym(n)}^{\geq\lambda} (M(\mu),M(\tau))/\Hom_{\Sym(n)}^{>\lambda}(M(\mu),M(\tau))$$
 is torsion free.
\end{lemma}

\q We can improve on this somewhat. A subset $\sigma$ of $\Par(n)$ will be  called co-saturated (also said to be a co-ideal)  if whenever $\lambda,\mu\in \sigma$, $\lambda\in \sigma$ and 
$\lambda \trianglelefteq \mu$ then $\mu\in \sigma$.

\begin{proposition} Let $\sigma,\tau$ be cosaturated subsets of $\Par(n)$ with the $\tau\subseteq \sigma$. Then,  for all $\mu,\nu\in \Par(n)$,  the quotient
$$ \Hom_{\Sym(n)}^{\sigma} (M(\mu),M(\nu))/\Hom_{\Sym(n)}^{\tau}(M(\mu),M(\nu))$$
 is torsion free.
\end{proposition}

\begin{proof}  If there is a co-saturated subset $\theta$ with $\tau\subset \theta\subset \sigma$ (and $\theta\neq \sigma,\tau$)  and if 
$$ \Hom_{\Sym(n)}^{\sigma} (M(\mu),M(\nu))/\Hom_{\Sym(n)}^{\theta}(M(\mu),M(\nu))$$
and
$$ \Hom_{\Sym(n)}^{\theta} (M(\mu),M(\nu))/\Hom_{\Sym(n)}^{\tau}(M(\mu),M(\nu))$$
are torsion free  then so is 
$$ \Hom_{\Sym(n)}^{\sigma} (M(\mu),M(\nu))/\Hom_{\Sym(n)}^{\tau}(M(\mu),M(\nu)).$$
Thus we are reduced to the case $\tau=\sigma\backslash \{\lambda\}$, where $\lambda$ is a maximal element of $\sigma$.  We choose a total order $\preceq$ on $\Par(n)$ refining $\leq$ such that, writing out the elements of $\Par(n)$ in descending order $\lambda^1\succ \lambda^2 \cdots \succ \lambda^h$ we have  $\tau=\{\lambda^1,\ldots,\lambda^k\}$, $\sigma=\{\lambda^1,\ldots,\lambda^{k+1}\}$ (so $\lambda=\lambda^{k+1}$) for some $k$. Then we have 
\begin{align*} \Hom_{\Sym(n)}^{\sigma}& (M(\mu),M(\nu))/\Hom_{\Sym(n)}^{\tau}(M(\mu),M(\nu)\cr
&= \Hom_{\Sym(n)}^{\succeq\lambda}(M(\mu),M(\nu))/\Hom_{\Sym(n)}^{\succ \lambda}(M(\mu),M(\nu)
\end{align*}
which is torsion free by the Lemma.
\end{proof}

\q Returning  to the general situation we have, by the Proposition and Section 3, (2), the following results.

\begin{corollary}  The $S$-module $V(\lambda)$ is torsion free. 
\end{corollary}

\begin{corollary} Let $\sigma$ be cosaturated set (with respect to $\leq$). Then $S(\sigma)$ is a pure submodule of $S$.
\end{corollary}

\section{Cosaturated $\Sym(n)$-sets}

\q     From Corollary 4.4,  if $\sigma$ is any co-saturated subset of $\Par(n)$ then we may identify $\que\otimes_\zed S(\sigma)$ with an $S_{\Omega,\que}$-submodule of $S_{\Omega,\que}$ via the natural map $\que\otimes_\zed S(\sigma)\to S_\que$.

\q We now suppose that $\Omega$ is cosaturated, by which we mean that $\zeta(\Omega)$ is a cosaturated subset of $\Par(n)$.  We check that much of  the structure, described by Green for the Schur algebras  in \cite{CASA},  still stands in this more general case.

\q Let $\sigma$ be a co-saturated subset of the support $\zeta(\Omega)$ of $\Omega$. Let $\mu\in \zeta(\Omega)$.  If $\nabla_\Omega(\mu)_\que$ is a composition factor of $S(\sigma)_\que$  then it is a composition factor of $S(\lambda)_\que$  and hence of $S_\que \xi_\lambda$, for some $\lambda\in \sigma$. Hence we have $\Hom_{\Sym(n)}(S\xi_\lambda,\nabla_\Omega(\mu)_\que)\neq 0$ and so $\mu\geq \lambda$, Lemma 3.1(ii),  and therefore  $\mu\in \sigma$. 

\q We   fix $\lambda\in \zeta(\Omega)$. Then $\Hom_{\Sym(n)}(S_\que \xi_\lambda,\nabla_\Omega(\lambda)_\que)={}^{\lambda}\nabla_\Omega(\lambda)_\que=\que$, by Lemma 3.1(i),  so that $\nabla_\Omega(\lambda)_\que$ is a composition factor of $S(\geq\lambda)_\que$, but not of $S(>\lambda)_\que$.   Now we can write $S(\geq\lambda)_\que=S(>\lambda)\oplus I$ for some ideal $I$  which,  as a left $S_\que$-module,  has only the composition factor $\nabla_\Omega(\lambda)_\que$. Hence $I$ is isomorphic to the  matrix algebra $M_d(\que)$, where $d=\dim \nabla_\Omega(\lambda)_\que$,  and, as a left $S_\que$-module $S(\geq\lambda)/S(>\lambda)$ is a direct sum of $d$ copies of $\nabla_\Omega(\lambda)_\que$.  Hence 
\begin{align*}\dim_\que  {}^\lambda V_\que(\lambda)&=\dim_\que  \Hom_{\Sym(n)}(S_\que \xi_\lambda, V_\que(\lambda))\cr
&=d \dim_\que \Hom_{\Sym(n)}(S_\que \xi_\lambda,\nabla_\Omega(\lambda)_\que)\cr
&=d \dim_\que {}^\lambda \nabla_\Omega(\lambda)_\que=d.
\end{align*}

Thus  $\dim V_\que(\lambda)^\lambda \otimes_\que {}^\lambda V_\que(\lambda)= \dim V_\que(\lambda)$ and  we have:
$$\hbox{ the natural map  }  V_\que(\lambda)^\lambda\otimes_\que {}^\lambda V_\que(\lambda) \to V_\que(\lambda)  \hbox{ is an isomorphism.}  \eqno{(1)}$$

\q We now consider the integral version.  We have the natural surjective map $ V(\lambda)^\lambda\otimes_\zed  {}^\lambda V(\lambda) \to V(\lambda)$.  
But the rank of $V(\lambda)^{\lambda}$ is the dimension of  $V_\que(\lambda)^{\lambda}$, the  rank of ${}^\lambda V(\lambda)$ is the dimension of  ${}^\lambda V_\que(\lambda)$, and the rank of  $V(\lambda)$ is the dimension of  $V_\que(\lambda)$ so that, by (1),  $ V(\lambda)^\lambda\otimes_\zed  {}^\lambda V(\lambda)$ and $V(\lambda)$ have the same rank.  Thus the surjective map  $ V(\lambda)^\lambda\otimes_\zed  {}^\lambda V(\lambda) \to V(\lambda)$ is an isomorphism. 

\q We have shown the following.

\begin{proposition}  Assume $\Omega$  is cosaturated. Then, for each $\lambda\in \Par(n)$,  the map 
$$V(\lambda)^{\lambda}\otimes_\zed  {}^\lambda V(\lambda) \to V(\lambda)$$
induced by multiplication in $S$, is an isomorphism.
\end{proposition} 

\begin{remark} If $k$ is a field then the corresponding algebras $S_{\Omega,k}$ over $k$ are Morita equivalent to those considered by Mathas  and Soriano in \cite{MathasSoriano}. There they  determined  blocks of such algebras (for the Schur algebras themselves this was done in \cite{DSchurfour}, and for the  quantised case by Cox  in \cite{Cox}).
\end{remark}

\section{Cellularity of  endomorphism  algebras of Young permutation modules}

\q We now establish our  main result, namely that the endomorphism algebra of a Young permutation module has the structure of a cellular algebra.   We first recall the notion  of a cellular algebra due to Graham and Lehrer, \cite{GL}.   (We have made some minor notational changes to be consistent with the notation above. The most serious of these is the reversal of the partial order from the definition given in \cite{GL}.)

\begin{definition}  Let $A$ be an algebra over a commutative ring $R$. A cell datum for   $(\Lambda^+,N,C,\,{}^*\,)$  for $A$ consists of the following.

\medskip

(C1) A partially ordered set $\Lambda^+$ and for each $\lambda\in \Lambda^+$ a finite set $N(\lambda)$ and an injective map $C: \coprod_{\lambda\in \Lambda^+} N(\lambda) \times N(\lambda) \to A$ with image an $R$-basis of $A$. 

(C2) For $\lambda\in \Lambda^+$ and $t,u\in N(\lambda)$ we write $C(t,u)=C^\lambda_{t,u}\in R$. Then   $^*$ is an $R$-linear anti-involution of $A$ such that $(C^\lambda_{t,u})^*= C^\lambda_{u,t}$.

(C3) If $\lambda\in \Lambda^+$ and $t,u\in N(\lambda)$ then for any element $a\in A$ we have 

$$aC^\lambda_{t,u}\equiv \sum_{t'\in N(\lambda)} r_a(t',t) C^\lambda_{t',u}   \hskip 20pt ({\rm mod } \  A(>\lambda))$$
where $r_a(t',t)\in R$ is independent of $u$ and where $A(>\lambda)$ is the $R$-submodule of $A$ generated by $\{ C^\mu_{t'',u''} \vert \mu\in  \Lambda^+, \mu>\lambda  \hbox{ and } t'',u''\in N(\mu)\}$.

\medskip

\q We say that $A$ is a cellular $R$-algebra if it admits a cell datum.
\end{definition}

\q Let $G$ be a finite group.  Let $\Omega$ be a finite $G$-set and let $R$ be a commutative ring.
Now  $G$ acts on $\Omega \times \Omega$. If $\A\subseteq \Omega\times \Omega$ is $G$-stable then we have an element  $a_\A\in \End_G(R\,\Omega)$ satisfying
$$a_\A(x)=\sum_y y$$
where the sum is over all $y\in \Omega$ such that $(y,x)\in \A$.  We write $\Orb_G(\Omega\times \Omega)$ for the set of $G$-orbits in $\Omega \times \Omega$. Then $\End_{R G}(R\,\Omega)$ free over $R$ on basis  $a_\A$, $\A\in \Orb_G(\Omega\times \Omega)$.  We have an involution on $\Omega\times \Omega$ defined by $(x,y)^*=(y,x)$, $x,y\in \Omega$.  For a $G$-stable subset $\A$ of $\Omega\times \Omega$ we write $\A^*$ for the $G$-stable set  $\{(x,y)^* \vert (x,y\in \Omega\}$.   

\q   For $\A, \B \in \Orb_G(\Omega\times \Omega)$ we have 
$$a_\A  a_\B=\sum_{\C \in \Orb_G(\Omega\times \Omega)} n^\C_{\A,\B}a_{\mathcal C}$$
where, for fixed $x\in \A$, $y\in \B$, the coefficient  $n^\C_{\A,\B}$ is the cardinality of the set $\{z\in\C \vert (x,z)\in \A \hbox{ and } (z,y)\in \B \}$.  It follows that $\End_{RG}(R\,\Omega)$ has an involutory anti-automorphism satisfying $a_{\DD}^*=a_{{\DD^*}}$,  for a $G$-stable subset $\DD$ of $\Omega\times \Omega$.  The notion of cellularity has built into it an involutory anti-automorphism ${}^*$  and in the case of endomorphism algebras of permutation modules,  we shall always use the one just defined.

\q We now restrict to the case $G=\Sym(n)$ with $\Omega$ a Young $\Sym(n)$-set as usual and  label by $\O_\alpha$, $\alpha\in \Lambda_\Omega$, the $G$-orbits in $\Omega$.  Now, for $\alpha\in \Lambda_\Omega$ and $x\in \Omega$ we have
$$\xi_\alpha(x)=\begin{cases} x, & \hbox{ if } x\in \O_\alpha; \cr
0, & \hbox{ otherwise.}
\end{cases}$$
Hence  $\xi_\alpha=a_\A$, where $\A=\{(x,x) \vert x\in \O_\alpha\}$ and therefore $\xi_\alpha^*=\xi_\alpha$. In particular we have $\xi_\lambda^*=\xi_\lambda$ for $\lambda\in \zeta(\Omega)$.   Thus we also have $S_{\Omega,R}(\sigma)^*=S_{\Omega,R}(\sigma)$, for $\sigma\subseteq \Par(n)$.

\q Note that if $\Gamma$ is a $G$-stable subset of $\Omega$ then we have the idempotent $e_\Gamma \in S_{\Omega,R}$ given on elements  of $\Omega$ by 
$$e_\Gamma (x)=\begin{cases} x,\hbox{ if } x\in \Gamma \, ;\cr
0, \hbox{ if } x\not\in \Gamma.
\end{cases}$$
Thus $e_\Gamma =a_\C$  where $\C=\{ (y,y)\vert y\in \Gamma \}$ and  $e_\Gamma^*=e_\Gamma$.

\q So now let $\Gamma$ be a Young $\Sym(n)$-set and let $\Omega$ be a co-saturated Young  $\Sym(n)$-set containing $\Gamma$.   We have the idempotent $e=e_\Gamma\in S_{\Omega,R}$ as above and $S_{\Gamma,R}=\End_{\Sym(n)}(R \Gamma)$ is naturally identified with $eS_{\Omega,R} e$.   

\begin{lemma} For $\lambda\in \zeta(\Omega)$ we have $e\nabla_\Omega(\lambda)_\que\neq 0$ if and only if  $\lambda\in \zeta^\trianglerighteq (\Gamma)$.

\end{lemma}

\begin{proof} We have $e=\sum_{\alpha\in \Lambda_\Gamma} \xi_\alpha$.  Hence $e \nabla_\Omega(\lambda)_\que\neq 0$ if and only if \\
$\xi_\alpha \nabla_\Omega(\lambda)_\que\neq 0$ i.e. $\sum_{\beta \in \Lambda_\Omega} \xi_\alpha \Hom_{\Sym(n)} (M_{\beta,\que},\Sp(\lambda)_\que)\neq 0$, for some $\alpha\in \Lambda_\Gamma$.  Hence $e\nabla_\Omega(\lambda)_\que\neq 0$ if and only if $\Hom_{\Sym(n)} (M_{\beta,\que},\Sp(\lambda)_\que)  \neq 0$ for some $\beta\in \Lambda_\Gamma$, i.e. if and only if $\Hom_{\Sym(n)}(M(\mu)_\que,\Sp(\lambda))\neq 0$ for some $\mu\in \zeta(\Gamma)$, i.e. if and only if  there exists $\mu\in \zeta(\Gamma)$ such that $\mu \trianglelefteq \lambda$.
\end{proof}

\q We fix a partial order $\leq$ on $\zeta(\Omega)$ refining the partial order $ \trianglelefteq $.

\q Let $\lambda\in \zeta(\Omega)$. We have the section  $V(\lambda)=S(\geq\lambda)/S(>\lambda)$ of $S=S_\Omega$.

\q We write $J^\op$ for the opposite ring of a ring $J$.  We write $S^\env$ for the enveloping algebra $S\otimes_\zed S^\op$.  We identify an $(S,S)$-bimodule with a left $S^\env$-module in the usual way.

\q We have the idempotent $\tilde e=e\otimes e\in S^\env$ and hence the Schur functor $\tilde f: \mod(S^\env)\to \mod( \tilde e S^\env \tilde e)$ as in \cite[Chapter 6]{EGS}. Moreover, \\
 $\tilde e S^\env \tilde e= eSe \otimes_\zed (eSe)^\op$.  Now $\tilde  f$ is exact so applying it  to the isomorphism $V(\lambda)^\lambda  \otimes_\zed {}^\lambda V(\lambda)\to V(\lambda)$ of Proposition 5.1  we obtain an isomorphism  
$$e \,  V(\lambda)^\lambda  \otimes_\zed {}^\lambda V(\lambda)\, e \to e V(\lambda)e \eqno{(1)}.$$

\q Now $\xi_\lambda S + S(>\lambda)=(S\xi_\lambda + S(>\lambda))^*$ so that $eV(\lambda)e\neq 0$ if and only if $e V(\lambda)^\lambda\neq 0$.  Moreover, $V(\lambda)^\lambda$ is a $\zed$-form of $\nabla(\lambda)_\que$ so that $eV(\lambda)e\neq 0$ if and only if $e\nabla_\Omega(\lambda)_\que\neq 0$.  Hence by,  Lemma 6.2,:

$$eV(\lambda)e\neq 0 \hbox{ if and only if }   \lambda\in \zeta^\trianglerighteq (\Gamma). \eqno{(2)}.$$

\q We now assemble our cell data.  We have the set $\Lambda^+=\zeta^\trianglerighteq(\Gamma)$ with partial order induced from the partial order $\leq $ on $\zeta(\Omega)$ (and also denoted $\leq$).   Let $\lambda\in \Lambda^+$. We let $n_\lambda=\dim_\que e\nabla_\Omega(\lambda)_\que$ and set $N(\lambda)=\{1,\ldots,n_\lambda\}$.  The rank of $eV(\lambda)^\lambda$ is  $n_\lambda$.  We choose elements $d_{\lambda,1},\ldots,d_{\lambda,n_\lambda}$ of $eS\xi_\lambda$ such that the elements $d_{\lambda,1}+S(>\lambda),\ldots, d_{\lambda,n_\lambda} + S(>\lambda)$ form a $\zed$-basis of $eV(\lambda)^\lambda= (eS\xi_\lambda + S(>\lambda))/S(>\lambda)$. Then $d_{\lambda,1}^*,\ldots,d_{\lambda,n_\lambda}^*$ are elements of $(eS\xi_\lambda)^*=\xi_\lambda Se$ and the elements $d_{\lambda,1}^*+ S(>\lambda), \ldots, d_{\lambda,n_\lambda}^* + S(>\lambda)$ form a $\zed$-basis of  ${}^\lambda V(\lambda)e =(\xi_\lambda S e + S(>\lambda))/S(>\lambda)$.  Thus $d_{\lambda,t}d_{\lambda,u}^*$ belongs to $eS\xi_\lambda Se$.  We define $C: \coprod_{\lambda\in \Lambda^+} N(\lambda)\times N(\lambda)\to eSe$ by $C(t,u)=C^\lambda_{t,u} =d_{\lambda,t}d_{\lambda,u}^*$, for $t,u\in N(\lambda)$.  

\q Let $M$ be the $\zed$-span of all $C^\lambda_{t,u}$, $\lambda\in \Lambda^+$, $t,u\in N(\lambda)$.   We claim that $M=eSe$.   We have $S=\sum_{\lambda\in \Lambda_\Omega} S\xi_\lambda S$ so that if the claim is false then there exists $\lambda\in \Lambda_\Omega$ such that $eS\xi_\lambda Se\not\subseteq M$. In that case we choose $\lambda$ minimal with this property. First suppose that $\lambda\not\in \zeta^\trianglerighteq (\Gamma)$. Then we have $eV(\lambda)e=0$, by (2),  i.e., $eS\xi_\lambda S e\subseteq S(>\lambda)$ and so $eS\xi_\lambda Se \subseteq eS(>\lambda)e$.  However,  $eS(>\lambda)e=\sum_{\mu>\lambda} eS\xi_\mu Se\subseteq M$, by minimality of $\lambda$  and so $eS\xi_\lambda S e\subseteq M$.  Thus we have $\lambda\in \Lambda^+=\zeta^\trianglerighteq (\Gamma)$. 

\q Now by (1)  the map 
$$(eS\xi_\lambda + S(>\lambda))\otimes_\zed (\xi_\lambda S e + S(>\lambda))\to eS\xi_\lambda Se + S(>\lambda)$$
induced by multiplication is surjective. Moreover we have $eS\xi_\lambda + S(>\lambda)=\sum_{t=1}^{n_\lambda} \zed d_{\lambda,t} +S(>\lambda)$ and  $\xi_\lambda Se + S(>\lambda)=\sum_{u=1}^{n_\lambda} \zed d_{\lambda,u}^* +S(>\lambda)$ so that 
$$eS\xi_\lambda Se \subseteq \sum_{t,u=1}^{n_\lambda} \zed d_{\lambda,t} d_{\lambda,u}^*+S(>\lambda)=\sum_{t,u=1}^{n_\lambda} \zed C^\lambda_{t,u}+S(>\lambda)$$
and hence 
$$eS\xi_\lambda Se\subseteq \sum_{t,u=1}^{n_\lambda} \zed C^\lambda_{t,u}+eS(>\lambda)e.$$
But now $\sum_{t,u=1}^{n_\lambda} \zed C^\lambda_{t,u}\subseteq M$ by definition and again $eS(>\lambda)e \subseteq M$ by the minimality of $\lambda$ so that $eS\xi_\lambda Se\subseteq M$ and the claim is established.

\q The elements $C^\lambda_{t,u}$, $\lambda\in \Lambda^+$, $1\leq t,u\leq n_\lambda$ form a spanning set  of $eS_\Omega e= S_{\Gamma}$.  But the rank of $eSe$ is the $\que$-dimension of $eS_\que e$, i.e., the $\que$-dimension of $S_{\Gamma,\que}$ and this is $\sum_{\lambda\in \Lambda^+} (\dim e \nabla_\Omega(\lambda))^2$ by Remark 2.2. Hence the elements $C^\lambda_{t,u}$, with $\lambda\in \Lambda^+$, $t,u\in N(\lambda)$, form a $\zed$-basis of $eSe$.

\q We have now checked the defining properties (C1) and (C2) of cell structure  and it remains to check (C3).  We fix $\lambda\in \Lambda^+$ and let $1\leq t,u\leq n_\lambda$.  Let $a\in eSe$. 
Then we have 
$$aC^\lambda_{t,u}=a d_{\lambda,t} d_{\lambda,u}^*.$$
Now we have $\sum_{i=1}^{n_\lambda} \zed d_{\lambda,i} + S(>\lambda)=eS\xi_\lambda + S(>\lambda)$ so we may write $ad_{\lambda,t}=\sum_{t'=1}^{n_\lambda} r_a(t',t) d_{\lambda,t'} + y $ for some integers $ r_a(t',t) $ and an element $y$ of $S(>\lambda)$.  Thus we  have 
\begin{align*}aC^\lambda_{t,u}=&a d_{\lambda,t} d_{\lambda,u}^*=\sum_{t'=1}^{n_\lambda} r_a(t',t) d_{\lambda,t'} d_{\lambda,u}^* + y d_{\lambda,u}^* \cr
&=\sum_{t'=1}^{n_\lambda} r_a(t',t) C^\lambda_{t',u} + y d_{\lambda,u}^*
\end{align*}
and hence 
$$aC^\lambda_{t,u}=\sum_{t'=1}^{n_\lambda} r_a(t',t) C^\lambda_{t',u}   \hskip 10pt (\hbox{mod}  \  S(>\lambda)).$$
 \q We have thus checked defining property (C3) and hence proved the following.

  \begin{theorem}   Let $\Gamma$ be a Young $\Sym(n)$-set. Then  $(\Lambda^+, N,C,{}^*)$ is a cell structure on $S_{\Gamma,\zed}=eS_{\Omega,\zed}e=\End_{\Sym(n)}(\zed \Gamma)$. 
  \end{theorem}
  
  \q One now obtains a cell structure on $\End_{\Sym(n)}(R \Gamma)$, for any commutative ring $R$ by base change.
  
  \q There  is also the question of when an endomorphism algebra over a field $k$  is quasi-hereditary.  If $k$ has characteristic $0$ then $\End_{\Sym(n)}(k\Gamma)$ is semisimple and there is nothing to consider. We assume now that the characteristic of $k$ is $p>0$.  By \cite[Remark 3.10]{GL} (see also \cite{KX1}, \cite{KX2}) $\End_{\Sym(n)}(k\Gamma)$  is quasi-hereditary if and only if the number of irreducible 
  $\End_{\Sym(n)}(k\Gamma)$-modules (up to isomorphism)  is equal to the length of the cell chain,  i.e., $|\zeta^\trianglerighteq(\Gamma)|$. By  Lemma 2.4 , the number of irreducible $\End_{\Sym(n)}(k\Gamma)$-modules  is  
  $|\zeta^{\trianglerighteq_p} (\Gamma)|$.  Moreover, we have   $\zeta^{\trianglerighteq_p} (\Gamma)\subseteq  \zeta^{\trianglerighteq} (\Gamma)$ and so  $\End_{\Sym(n)}(k\Gamma)$ is quasi-hereditary if and only if $\zeta^{\trianglerighteq } (\Gamma)\subseteq  \zeta^{\trianglerighteq_p} (\Gamma)$. We spell this out in the following result.

\begin{theorem} Let $k$ be a field of  characteristic $p>0$ and let $\Gamma$ be a Young $\Sym(n)$-set. Then the endomorphism algebra $\End_{\Sym(n)}(k\Gamma)$ of the permutation module $k\Gamma$ is quasi-hereditary if and only if for every partition $\lambda$ of $n$ such that the Young subgroup $\Sym(\lambda)$ appears as the stabiliser of a point  of $\Gamma$ and every partition $\mu \trianglerighteq \lambda$ there exists a partition $\tau$ such that $\Sym(\tau)$ appears as a point stabiliser and  such that $\mu$ $p$-dominates $\tau$, i.e.,  there exists a weak $p$ expansion  $\tau=\sum_{i\geq 0} p^i\gamma(i)$, with $\gamma(i)\in \Lambda(n)$, and ${\overline {\gamma(i)}}\trianglelefteq \mu(i)$ for all $i$ (where $\mu=\sum_{i\geq 0} p^i\mu(i)$ is the base $p$-expansion of $\mu$ and where ${\overline {\gamma(i)}}$ is the partition obtained by writing the parts of $\gamma(i)$ in descending order, for $i\geq 0$).
\end{theorem}

\begin{remark} We emphasise  that the above gives a criterion for the endomorphism algebra $\End_{\Sym(n)}(k\Gamma)$ of the Young permutation module $k\Gamma$ to be quasi-hereditary with respect to any labelling of the simple modules by a partially ordered set (which may have nothing to do with those considered above) thanks to the result of K\"onig and Xi, \cite[Theorem 3.]{KX2}. Thus if $\Gamma$ does not satisfy the condition above then $S_{\Gamma,k}$ can not have finite global dimension by \cite[Theorem 3]{KX2} and hence is not quasi-hereditary. 
 \end{remark}

\section{Example: Tensor Powers}

\q Let $R$ be a commutative ring and let $E_R$ be a free $R$-module on basis $e_{1,R},\ldots,e_{n,R}$.  Let $r$ be a positive integer and let  $I(n,r)$ be the set described in Example 3.2.   Then the $r$-fold tensor product $E_R^{\otimes r}=E_R\otimes_R \otimes \cdots \otimes_R E_R$ has $R$-basis $e_{i,R}=e_{i_1,R}\otimes\cdots e_{i_r,R}$, $i\in I(n,r)$,  and we thus identify  $E_R^{\otimes r}$ with $R I(n,r)$, the free $R$-module on $I(n,r)$.  

\begin{remark} The symmetric group $\Sym(r)$ acts on $E_R^{\otimes r}$ by place permutations, i.e. $w\cdot e_{i,R}=e_{i\circ w^{-1},R}$, for $w\in \Sym(r)$, $i\in I(n,r)$.  Thus we may regard $E_R^{\otimes r}$ as the permutation module $R I(n,r)$, with $\Sym(r)$, acting on $I(n,r)$ by $w\cdot i=i\circ w^{-1}$.   The endomorphism algebra $\End_{\Sym(r)}(E^{\otimes r}_R)$ is the Schur algebra $S_R(n,r)$.

\q The stabiliser of $i\in I(n,r)$ is the direct product of the symmetric groups on the fibres of $i$ (regarded as a subgroup of $\Sym(r)$ in the usual way). Hence $I(n,r)$ is a Young $\Sym(r)$-set. Hence $E_R^{\otimes r}$ is a Young permutation module and hence $S_R(n,r)$ is cellular.  Moreover, $\zeta(I(n,r))$ is the set $\Lambda^+(n,r)$ of all partitions of $r$ with at most $n$ parts. This is a co-saturated set and hence for a prime $p$   we have $\zeta(I(n,r))=\zeta^{\trianglerighteq}(I(n,r))=\zeta^{\trianglerighteq_p}(I(n,r))$.  Hence, for a field $k$ of characteristic $p$ the Schur algebra  $S_k(n,r)$ is  quasi-hereditary.  

\q However, this is not a new proof since our treatment relies crucially  on  a detail from Green's analysis of $S_\zed(n,r)$ as in \cite{CASA}, at least in the case $n=r$.  (See Example 3.2 above and the  proofs of the results of Section 4.)
\end{remark}

\q We now regard $E_R$ as an $R\Sym(n)$-module  with $\Sym(n)$ permuting the basis $e_{1,R},\ldots,e_{n,R}$ in the natural way.  This  action induces an action on the tensor product $E_R^{\otimes r}$. Specifically, we have $w\cdot e_{i,R}=e_{w\circ i,R}$, for $w\in \Sym(n)$, $i\in I(n,r)$, and we thus regard $E^{\otimes r}_R$ as the permutation module $R I(n,r)$.  For $w\in \Sym(n)$, $i\in I(n,r)$ we have $w\circ i=i$ if and only if $w$ acts as the identity on the image of $i$, so that the stabiliser of $i$ is the group of symmetries of the complement of the image of $i$ in $\{1,\ldots,n\}$, identified  with a subgroup of $\Sym(n)$ in the usual way. Thus  $I(n,r)$ is a Young $\Sym(n)$-set so we have the following consequence of   Theorem 6.3, answering  a question raised in \cite{BDM}.

\begin{proposition}  The endomorphism algebra \\
$\End_{\Sym(n)}( E^{\otimes r}_R)
=\End_{\Sym(n)}(R I(n,r))$ is a cellular algebra. 
\end{proposition}

\q The  support of $I(n,r)$ consists of hook partitions, more precisely we have 
$$\zeta(I(n,r))=\{(a,1^b) \vert a+b=n, 1\leq b \leq r\}.$$

\q Hence we have 
$$\zeta^{\trianglerighteq}(I(n,r))=\{\lambda=(\lambda_1,\lambda_2,\ldots)\in \Par(n) \vert \lambda_1\geq n-r\}.$$
Let $k$ be a field of characteristic $p>0$. Then $\End_{\Sym(n)}(E^{\otimes r}_k)$ is quasi-hereditary if and only if $\zeta^{\trianglerighteq}(I(n,r))\subseteq \zeta^{\trianglerighteq_p}(I(n,r)$, i.e., if and only for every  $\mu=(\mu_1,\mu_2,\ldots) \in \Par(n)$ with $\mu_1\geq n-r$ there exists some  $\lambda=(a,1^b)$, $1\leq b\leq r$, such that $\lambda \trianglelefteq_p \mu$.

\q We are able to give an explicit list of  quasi-hereditary algebras arising in the above manner.

\begin{proposition} Let $k$ be a field of characteristic $p>0$. Let $n$ be a positive integer and $E$ an $n$-dimensional $k$-vector space with basis $e_1,\ldots,e_n.$  We regard $E$ as a $k\Sym(n)$-module with $\Sym(n)$ permuting the basis in the obvious way. For $r\geq 1$we regard the  $r$th tensor power $E^{\otimes r}$ as a $k\Sym(n)$-module via the usual tensor product action. Then $\End_{\Sym(n)}(E^{\otimes r})$ is quasi-hereditary if and only if:

(i) $p$ does not divide $n$; and 

(ii) either $n<2p$ (and  $r$ is arbitrary)    or  $n>2p$ and $r<p$. 

\end{proposition}

\begin{proof}

\q We see this in a number of steps.  We regard $E^{\otimes r}$ as the permutation module $k I(n,r)$, as above, with $\Sym(n)$ action by $w\cdot i= w\circ i$, for $w\in\Sym(n)$, $i\in I(n,r)$. We shall say that $I(n,r)$ is quasi-hereditary if $\End_{\Sym(n)}(E^{\otimes r})$ is. 

\bs

{\it Step 1.}     If $p$ divides $n$ then $I(n,r)$ is not quasi-hereditary.

\q We have $(n-1,1)\in \zeta(I(n,r))$ and $(n,0) \trianglerighteq (n-1,1)$ so that $(n,0)\in \zeta^{ \trianglerighteq}(I(n,r))$.  Now $n=pm$, for some positive integer $m$,   so that $\mu=(n,0)=p(m,0)$ has  base $p$ expansion $(n,0)=\sum_{i\geq 0} p^i \mu(i)$, with  restricted part $\mu(0)=0$. Thus if $\tau=(a,1^b)$ has weak $p$-expansion $\tau=\sum_{i\geq 0} p^i \gamma(i)$ and ${\overline{ \gamma(i)}}  \trianglelefteq \mu(i)$, for all $i$, then $\gamma(0)=0$ and $\tau$ is divisible by $p$. However, this is not the case so no such weak $p$-expansion exists and $\mu\in \zeta^{ \trianglerighteq}(I(n,r))\backslash  \zeta^{ \trianglerighteq_p}(I(n,r))$.  Thus $ \zeta^{ \trianglerighteq}(I(n,r))\neq \zeta^{ \trianglerighteq_p}(I(n,r))$ and $I(n,r)$ is not quasi-hereditary.

\bs 

{\it Step 2.}    If $p$ does not divide $n$ then $I(n,1)$ is quasi-hereditary.

\q We have $\zeta(I(n,1))=\{(n-1,1)\}$.  If  $\mu\in  \zeta^{ \trianglerighteq}(I(n,1))\backslash  \zeta^{ \trianglerighteq_p}(I(n,r))$ then 
$\mu=(n,0)$. 
Now $n$ has base $p$ expansion $n=\sum_{i\geq 0} p^i n_i$, with $0\leq n_i<p$ for all $i\geq 0$  and  $n_0\neq 0$ and $\mu$ has base $p$ expansion $\mu=\sum_{i\geq 0} p^i \mu(i)$, with $\mu(i)=(n_i,0)$, for all $i\geq 0$. 

\q  But now we write 
$$\tau=(n-1,1)=(n_0-1,1)+\sum_{i\geq 1} p^i (n_i,0)$$
and $\tau$ has weak $p$-expansion $\tau=\sum_{i\geq 0} p^i \gamma(i)$, with $\gamma(0)=(n_0-1,1)$, $\gamma(i)=(n_i,0)$ for $i\geq 1$. Moreover   ${\overline {\gamma(i)}}\leq \mu(i)$, for all $i$ so that $(n,0)\in  \zeta^{ \trianglerighteq_p}(I(n,1))$.  Thus $ \zeta^{ \trianglerighteq}(I(n,1))= \zeta^{ \trianglerighteq_p}(I(n,1))$ and $I(n,1)$ is quasi-hereditary.

\bs

{\it Step 3.}    If $\mu\in  \zeta^{ \trianglerighteq}(I(n,r))$ is $p$-restricted then $\mu\in  \zeta^{ \trianglerighteq_p}(I(n,r))$ 

\q We have $\mu \trianglerighteq (a,1^b)$  for some $n=a+b$, $1\leq b\leq r$.  The partition  $\mu$ has base $p$ expansion $\mu=\sum_{i\geq 0} p^i \mu(i)$, with $\mu(i)=0$ for all $i\geq 1$. 

\q But now $\tau=(a,1^b)$ has week $p$-expansion $\tau=\sum_{i\geq 0} p^i \gamma(i)$, with $\gamma(0)=(a,1^b)$ and 
$\gamma(i)=0$ for all $i\geq 1$.  Furthermore we have  ${\overline {\gamma(i)}} \trianglelefteq  \mu(i)$ for all $i\geq 0$ so $\mu\in   \zeta^{ \trianglerighteq_p}(I(n,r))$.

\bs

{\it Step 4.}   If $n<p$ then $I(n,r)$ is quasi-hereditary.

\q This follows from Step 3 all since elements of $\Par(n)$ are restricted.

\bs

{\it Step 5.}    If $p<n<2p$ then $I(n,r)$ is quasi-hereditary. 

\q For a contradiction suppose not and let \\
$\mu=(\mu_1,\mu_2,\ldots) \in  \zeta^{ \trianglerighteq}(I(n,r))\backslash  \zeta^{ \trianglerighteq_p}(I(n,r))$.   We have $\mu \trianglerighteq (a,1^b)$  for some $a,b$ with $n=a+b$, $1\leq b\leq r$.  Choose $a,b$ with this property with $b\geq 1$ minimal.  If $b=1$ then $\mu\in  \zeta^{ \trianglerighteq}(I(n,1))$, which by Step 2 is   
$ \zeta^{ \trianglerighteq_p}(I(n,1))$.  Thus we have  $b\geq 2$. 

\q We claim that $\mu_1=a$. Since $\mu \trianglerighteq (a,1^b)$ the length $l$, say, of $\mu$ is at most the length of $(a,1^b)$, i.e. $b+1$.  Put $\xi=(\xi_1,\xi_2,\ldots)=(a+1,1^{b-1})$.   If $\mu_1>a$ then $\mu_1\geq \xi_1$ and, for $1<i\leq l$, we have 
$$\mu_1+\cdots+\mu_i\geq a+1+(i-1)=a+i=\xi_1+\cdots+\xi_i.$$
  So $\mu  \trianglerighteq \xi=(a+1,1^{b-1})$, which is a contradiction, and the claim is established.

\q Note that $\mu$ is non-restricted, by Step 3, and,  since $\mu$ is a partition of $n<2p$ in the base $p$ expansion $\mu=\sum_{i\geq 0} p^i\mu(i)$ of $\mu$,  we must have $\mu(1)=(1,0)$ and $\mu(i)=0$ for $i\geq 2$.  Let $\tau=(a,1^b)$. Then $\tau  \trianglelefteq \mu$ implies that $\tau -(p,0)  \trianglelefteq \mu -(p,0)=\mu(0)$. But now 
$$\tau=(a,1^b)=(a-p,1^b) + p(1,0)$$
so we have the weak $p$ expansion $\tau=\sum_{i\geq 0} p^i \gamma(i)$ with $\gamma(0)=(a-p,1^b)$, $\gamma(1)=(1,0)$ and $\gamma(i)=0$ for $i > 1$.  Since ${\overline {\gamma(i)}} \trianglelefteq \mu(i)$ for all $i\geq 0$ we have $(a,1^b) \trianglelefteq_p \mu$ and so $\mu\in  \zeta^{ \trianglerighteq_p}(I(n,r))$, a contradiction. 

\bs

{\it Step 6.}    If $n>2p$ and $r\geq p$ then $I(n,r)$ is not quasi-hereditary.

\q Note that $\zeta(I(n,r))$ contains $(n-p,1^p)$ and hence $\zeta^{ \trianglerighteq }(I(n,r))$ contains $\mu=(n-p,p)$.  Now we have $\mu=(n-2p,0)+p(1,1)$ and so $\mu=\mu(0)+p\xi$, where $\mu(0)$ has at most one part and  $\xi$ has two parts.  Hence in the base $p$ expansion $\mu=\sum_{i\geq 0} p^i \mu(i)$,  there is for some $j\geq 1$,  such that   $\mu(j)$ has  two parts. 

\q Now if $\mu\in  \zeta^{ \trianglerighteq_p }(I(n,r))$ there there exists some $\tau=(a,1^b)$ with weak $p$ expansion $\tau=\sum_{i\geq 0} p^i \gamma(i)$ such that  ${\overline {\gamma(i)}}\trianglelefteq \mu(i)$ for all $i\geq 0$.  But then $\gamma(j)$ must have at least two parts. Since $j\geq 1$,  the partition $\tau=(a,1^b)$ has two parts of size at least $p$. This is not the case so there is no such weak $p$ expansion and  $\mu\not \in  \zeta^{ \trianglerighteq_p }(I(n,r))$.   Thus $ \zeta^{ \trianglerighteq}(I(n,r))\neq \zeta^{ \trianglerighteq_p}(I(n,r))$ and $I(n,r)$ is not quasi-hereditary.

\bs

{\it Step 7.}    If $n>2p$,  if  $p$ does not divide $n$ and if $r<p$,  then $I(n,r)$ is quasi-hereditary.

\q If not there exists $\mu=(\mu_1,\mu_2,\ldots) \in  \zeta^{ \trianglerighteq}(I(n,r))\backslash  \zeta^{ \trianglerighteq_p}(I(n,r))$.  Thus $\mu \trianglerighteq (a,1^b)$, for some $n=a+b$, $b\geq 1$ and, as in Step 5, we choose such $(a,1^b)$ with $b$ minimal. Again, by Step 2, we have $b\geq 2$. 

\q We claim that $\mu_1=a$. If not, we get $\mu\trianglerighteq  (a+1,1^{b-1})$ as in Step 5, contradicting the minimality of $b$. 

\q Thus we have $\mu_2+\cdots+\mu_n=n-\mu_1=b<p$, in particular we have $\mu_i<p$ for all $i\geq 1$.  Hence in  the base $p$ expansion $\mu=\sum_{i\geq 0} p^i \mu(i)$,  for all $i\geq 1$   we have $\mu(i)=(c_i,0,\ldots,0)$, for some $0\leq c_i<p$.  Also, $\mu(0)=(k,\mu_2,\ldots,\mu_n)$, for some $k >  0$. 

\q Now we have
$$\tau=(a,1^b)=(k+\sum_{i\geq 1}  p^ic_i, 1^b)=(k,1^b)+\sum_{i\geq 1} p^i (c_i,0,\ldots,0).$$
Thus  we have the weak $p$-expansion $\tau=\sum_{i\geq 0} p^i \gamma(i)$, with $\gamma(0)=(k,1^b)$ and $\gamma(i)=(c_i,0,\ldots,0)$, for $i\geq 1$. Furthermore, ${\overline { \gamma(i)}}\trianglelefteq \mu(i)$, for all $i\geq 0$ so that $\mu\in  \zeta^{ \trianglerighteq_p}(I(n,r))$ and therefore  $\zeta^{ \trianglerighteq}(I(n,r))=  \zeta^{ \trianglerighteq_p}(I(n,r))$ and $I(n,r)$ is quasi-hereditary.
\end{proof}

\q Let $k$ be a field. Recall that, for $\de\in k$, and $r$ a positive integer we have the partition algebra $P_r(\de)$ over $k$.    One may find a detailed account of the construction and properties of $P_r(\de)$ in for example the papers by  Paul P. Martin, \cite{PPM1}, \cite{PPM2}, and \cite{HR}, \cite{BDM}.  Suppose now that $k$ has characteristic $p>0$ and $\de=n1_k$, for some positive integer $n$.  Let $E_n$ be an $n$-dimensional vector space with basis $e_1,\ldots,e_n$. Then $P_r(n)=P_r(n 1_k)$ acts on $E_n^{\otimes r}$.     By a result of Halverson-Ram, \cite[Theorem 3.6]{HR} the image of the representation $P_r(n)\to \End_k(E_n^{\otimes r})$  is $\End_{\Sym(n)}(E^{\otimes r})$.  Moreover, for $n\gg 0$ the action of $P_r(n)$ is faithful.  Let  $N=n+ps$, for $s$ suitably large, so that $P_r(n)=P_r(N)$ acts faithfully on $E_N^{\otimes r}$.  Thus $P_r(n)$ is quasi-hereditary if and only if $\End_{\Sym(N)}(E_N^{\otimes r})$ is faithful. Hence from Proposition 7.3 we have the following, which is a special case of a result of  K\"onig and  Xi, \cite[Theorem 1.4]{KX2}.

\begin{corollary} The partition algebra $P_r(n)$ is quasi-hereditary if and only if $n$ is prime to $p$ and  $r<p$. 
\end{corollary}

\bs\bs\bs 



\end{document}